\documentclass[12pt,leqno]{amsart}
\makeatletter
\usepackage{amsfonts,delarray,amssymb,amsmath,amsthm,a4,a4wide}

\usepackage{color}

\title[The mean curvature at the first singular time  ]{The mean curvature at the first singular time of the mean curvature flow}

\author{Nam Q.  Le}

\address{Department of
Mathematics, Columbia University, New York,
 USA}
\email{namle@math.columbia.edu}

\author{Natasa Sesum$^{*}$}
\address{Department of Mathematics, 
University of Pennsylvania, Philadelphia, PA,
USA}
\email{natasas@math.upenn.edu}

\thanks{$*:$ Partially supported
by NSF grant 0604657}

\newcommand{\review}[2][\right]{\relax
\ifx#1\right\relax \left.\fi#2#1\rvert}
\let\abs=\envert
 \newtheorem{definition}{Definition}[section]
\newtheorem{theorem}{Theorem}[section]

\newtheorem{remark}{Remark}[section]

\newtheorem{lemma}{Lemma}[section]
\newtheorem{cor}{Corollary}[section]
\newtheorem{claim}{Claim}[section]
\newtheorem{conj}{Conjecture}[section]

\newcommand{\bef}{\begin{flushright}}
\newcommand{\eef}{\end{flushright}}

\newcommand{\eval}[2][\right]{\relax
\ifx#1\right\relax \left.\fi#2#1\rvert}

\let\abs=\envert

\numberwithin{equation}{section}
\let\norm=\enVert
\newcommand\e{\varepsilon}
\newcommand{\rem}{\mathrm{Rm}}
\newcommand{\h}{\hspace*{.24in}}

\newcommand{\vol}{\mathrm{vol}}
\newcommand{\ric}{\mathrm{Ric}}
\def\h{\hspace*{.24in}}

\def\beq{\begin{eqnarray*}}
\def\eeq{\end{eqnarray*}}

\def\RR{\mbox{$I\hspace{-.06in}R$}}

\begin{document}
\date{January 15, 2010}
\maketitle
\author
\pagenumbering{arabic}
\begin{abstract}
Consider a family of smooth immersions $F(\cdot,t): M^n\to \mathbb{R}^{n+1}$ of closed hypersurfaces in $\mathbb{R}^{n+1}$ moving 
by the mean curvature flow $\frac{\partial F(p,t)}{\partial t} = -H(p,t)\cdot \nu(p,t)$, for $t\in [0,T)$.  We prove that the mean curvature blows up 
at the first singular time $T$ if all singularities are of type I.  In the case $n = 2$, regardless of the type of 
a possibly forming singularity, we show that at the first singular time the mean curvature necessarily blows up 
provided that either the Multiplicity One Conjecture holds or the Gaussian density is less than two. We 
also establish and give several applications of 
a local regularity theorem which is a parabolic analogue of Choi-Schoen
estimate for minimal submanifolds. 
\end{abstract}
\noindent

\section{Introduction}
Let $M^{n}$ be a compact $n$-dimensional hypersurface without boundary, and let $F_{0}: M^{n}\rightarrow \RR^{n+ 1}$ be a smooth immersion of $M^{n}$ into $\mathbb{R}^{n+1}$. Consider a smooth one-parameter family of immersions
\begin{equation*}
 F(\cdot, t): M^{n}\rightarrow \RR^{n +1}
\end{equation*}
satisfying
$
 F(\cdot, 0) = F_{0}(\cdot)$ 
and 
\begin{equation}
 \frac{\partial F(p, t)}{\partial t} = -H(p, t)\nu(p,t), \,\,\,~\forall (p, t)\in M\times [0, T).
\label{MCF1}
\end{equation}
Here $H(p, t)$ and $\nu(p, t)$ denote the mean curvature and a choice of unit normal for the hypersurface $M_{t} = F(M^{n},t)$ at $F(p, t)$, respectively.
We will sometimes also write $x(p, t) = F(p, t)$ and refer to (\ref{MCF1}) as to the mean curvature flow equation. Furthermore, for any 
compact $n$-dimensional hypersurface $M^{n}$ which is smoothly embedded in $\RR^{n+1}$ by $F: M^{n}\rightarrow \RR^{n+1}$, 
let us denote by $g = (g_{ij})$ the induced metric, $A = (h_{ij})$ the second fundamental form,  $d\mu =\sqrt{\text{det}~(g_{ij})}~dx$ the volume form,
$\nabla$ the induced Levi-Civita connection.
Then the mean curvature of $M^{n}$ is given by $H = g^{ij}h_{ij}.$\\
\h Without any special assumptions on $M_0$, the mean curvature flow (\ref{MCF1}) will in general develop singularities  in finite time, characterized by a blow up of the second fundamental form $A(\cdot,t)$.

\begin{theorem}[Huisken \cite{Huisken84}]
Suppose $T < \infty$ is the first singularity time  for a compact mean curvature flow. Then $\sup_{M_t}|A|(\cdot,t) \to \infty$ as $t\to T$.
\label{Aunbound}
\end{theorem}

By the work of Huisken and Sinestrari \cite{HS} the blow up of $H$ near a singularity is known for mean convex hypersurfaces. They show that 
when $H \ge 0$ one has a pinching curvature estimate stating that $|A|^2 \le C_1 H^2 + C_2$, for uniform constants $C_1, C_2$.  In \cite{Sm}  a similar pinching estimate has been proven for star shaped hypersurfaces. The present article 
establishes the blow up of the mean curvature  in the case of  type I singularities.  
\begin{definition}
We say that the mean curvature flow (\ref{MCF1}) develops a singularity of type I at $T < \infty$ if the blow-up rate of the curvature satisfies an upper bound of the form
\begin{equation}
 \mathrm{max}_{M_{t}} \abs{A}^2(\cdot, t) \leq \frac{C_{0}}{T-t}, ~ 0\leq t<T.
\label{typeI}
\end{equation}
\end{definition}
In this paper, we prove the following
\begin{theorem}
Assume (\ref{typeI}) for the mean curvature flow (\ref{MCF1}). If
\begin{equation}
 \mathrm{max}_{M_{t}} \abs{H}^2(\cdot, t) \leq C_{0}
\label{Hbound}
\end{equation}
then the flow can be extended past time $T$.
\label{MCbound}
\end{theorem}
In fact, the above theorem is a consequence of the following result.
\begin{theorem}
Assume (\ref{typeI}) for the mean curvature flow (\ref{MCF1}). If for some $\alpha \ge n + 2$
\begin{equation}
 \norm{H}_{L^{\alpha}(M\times [0, T))} \leq C_{0}
\label{intHbound}
\end{equation}
then the flow can be extended past time $T$.
\label{intMCbound}
\end{theorem}
The proofs of Theorems \ref{MCbound} and \ref{intMCbound} are based on blow-up arguments using Huisken's monotonicity formula, the classification of self-shrinkers
and White's local regularity theorem for mean curvature flow. 
\begin{remark}
 To some extent, the condition $\alpha \geq n +2$ appearing in Theorem \ref{intMCbound} is optimal as illustrated by the mean curvature flow of the
standard sphere $S^{n}$.
\end{remark}
Our Theorems \ref{MCbound} and \ref{intMCbound} left open the question on the possible blow up of the mean curvature at the first
singular time $T$ for mean curvature flows with singularities other than Type $I$. This seems to be a difficult question. However,
assuming the validity of Multiplicity One Conjecture (see page 7 of \cite{I1} and the precise statement
in Conjecture \ref{MultOne} of the present article), we prove the following
\begin{theorem}
 Let $M^{2}$ be a compact, smooth and embedded 2-dimensional manifold in $\RR^{3}$. If
\begin{equation}
 \mathrm{max}_{M_{t}} \abs{H}^2(\cdot, t) \leq C_{0}
\label{Hbound3d}
\end{equation}
then the flow can be extended past time $T$.
\label{3dcase}
\end{theorem}
The next result is independent of the Multiplicity One Conjecture. It is in some sense a refinement of  White's local regularity theorem \cite{White}.  White gives uniform curvature bounds in regions of spacetime where the Gaussian  density is close to one. We prove the following.
\begin{theorem}
 Let $M^{2}$ be a compact, smooth and embedded 2-dimensional manifold in $\RR^{3}$. Suppose that (\ref{Hbound3d}) holds. 
Let $y_{0}\in \RR^{3}$ be a point 
reached by the mean curvature 
flow (\ref{MCF1}) at time $T$. If 
\begin{equation}
 \lim_{t\nearrow T} \int \rho_{y_{0}, T} d\mu_{t}:= \lim_{t\nearrow T} \int \frac{1}{[4\pi(T-t)]^{n/2}} \mathrm{exp}(-\frac{\abs{y-y_{0}}^2}{4(T-t)})
d\mu_{t} < 2.
\label{below2}
\end{equation}
 then $(y_{0}, T)$ is a regular point of the mean curvature flow (\ref{MCF1}).
\label{3dreg}
\end{theorem}
\begin{remark}
Our theorem says that for mean curvature flow of surfaces with  Gaussian density $\lim_{t\nearrow T} \int \rho_{y_{0}, T} d\mu_{t}$ below 2, for every $y_0$  reached by the flow at time $T$, the mean
curvature must blow up at the first singular time. 
 In \cite{St}, Stone calculated the Gaussian density on spheres and cylinders. On spheres,
the density is $4/e\approx 1. 47$ and on cylinders it is $\sqrt{2\pi/e}\approx 1. 52.$
\end{remark}

We also give the following characterization of a finite time singularity of (\ref{MCF1}) that works in all dimensions $n \ge 2$.

\begin{theorem}
Assume that for the mean curvature flow (\ref{MCF1}), we have the following integral bound on the second fundamental form
\begin{equation}
 \norm{A}_{L^{p, q}(M\times [0, T))}: = \left(\int_{0}^{T}\left(\int_{M_{t}}\abs{A}^{q} d\mu\right)^{p/q} dt\right)^{1/p}< \infty
\label{intboundA}
\end{equation}
where $p, q\in (0, \infty)$ satisfy \begin{equation*}
 \frac{n}{q} + \frac{2}{p} =1.
\end{equation*}
\h Then the flow can be extended past time $T$.
\label{SFbound}
\end{theorem}

The previously mentioned results were all global characterizations ensuring that the flow can not develop any singularities 
as long as some global quantities are bounded uniformly in time. We also give a result regarding the local regularity theory.

\begin{theorem}
 Suppose $\mathcal{M} = (M_{t})$ is a smooth, properly embedded
solution of the mean curvature flow in $B(x_{0}, \rho)\times (t_{0}-\rho^2, t_{0})$ which reaches $x_{0}$ at time $t_{0}$. 
There exists $\e_{0} = \e_{0}(M_{0})>0$ such that if $0<\sigma\leq \rho$ and
\begin{equation}
 \int_{t_{0}-\sigma^2}^{t_{0}} \int_{M_{t}\cap B(x_{0},\sigma)}\abs{A}^{n +2} d\mu dt <\e_{0}
\end{equation}
then
\begin{equation}
 \max_{0\leq\delta\leq \sigma/2} \sup_{t\in [t_{0}-(\sigma-\delta)^2,t_{0})} \sup_{x\in B(x_{0}, \sigma-\delta)\cap M_{t}} 
\delta^2 \abs{A}^2 (x,t) < \e_{0}^{\frac{-2}{n +2}} 
(\int_{t_{0}-\sigma^2}^{t_{0}} \int_{M_{t}\cap B(x_{0},\sigma)}\abs{A}^{n +2} d\mu dt)^{\frac{2}{n +2}}.
\label{CSMCF}
\end{equation}
\label{CS}
\end{theorem}
\h Our theorem is a parabolic version of Choi-Schoen estimate \cite{CS} for minimal surfaces. Related results can be found in Ecker \cite{EckerCalvar}.
The precise estimate of the form (\ref{CSMCF}) for the case of minimal submanifolds can be found in  Shen-Zhu \cite{SZ}, Proposition 2.2 (see 
also \cite{CMMS}). Moreover, in \cite{LS, XYZ1, XYZ2}, the authors showed that if the $L^{n +2}$ norm in space-time of the second
fundamental form (or the mean curvature but under various convexity assumptions) is finite then it is possible to extend the mean curvature flow
beyond the time interval under consideration. Our theorem can be viewed as a local version of these results without 
imposing any convexity assumptions. It turns out that the conclusion of Theorem \ref{CS} also 
holds in the case when the ambient space is a complete Riemannian manifold with bounded geometry.  We show that in Corollary \ref{cor-ext}.

The organization of the paper is as follows. In section \ref{sec-typeI} we give the proofs of Theorems \ref{MCbound} 
and \ref{intMCbound}.  In section \ref{3d} we prove Theorems \ref{3dcase} and \ref{3dreg}. The proof of Theorem \ref{SFbound}
will be given in section \ref{sec-global}. We conclude the paper with section \ref{sec-local} in which we prove  Theorem \ref{CS} and give some applications to it.

{\bf Acknowledgements:} The authors would like to thank Rick Schoen and Bill Minicozzi for helpful discussions. Minicozzi pointed out to us to try to prove Theorem \ref{3dcase} under the Multiplicity One Conjecture.

\section{Characterization of type I singularities}
\label{sec-typeI}
This section is concerned with the proofs of Theorems \ref{MCbound} and \ref{intMCbound}.
\begin{proof}[Proof of Theorem \ref{MCbound}]
Without loss of generality, assume that $M^{n}\subset B_{1}(0) \subset \RR^{n+1}$. Let $y_{0}\in \RR^{n+1}$ be a point reached by the mean curvature 
flow (\ref{MCF1}) at time $T$, that is, there exists a sequence $(y_{j}, t_{j})$ with $t_{j}\nearrow T$ so that $y_{j}\in M_{t_{j}}$ and 
$y_{j}\rightarrow y_{0}$. We show that $(y_{0}, T)$ is a regular point of (\ref{MCF1}). \\
Note that the distance estimate (\cite{Ecker}, Corollary 3.6) gives
\begin{equation}
\mathrm{dist} (M_{t}, y_{0}) \leq \sqrt{2n(T-t)}, ~\text{for}~ t<T. 
\label{dist}
\end{equation}

Consider the parabolic dilation $D_{\lambda}: \RR^{n+1} \times [0, T) \to \RR^{n+1}\times [-\lambda^2 T, 0)$ of scale $\lambda >0$ at $(y_{0}, T)$ 
defined by
\begin{equation}
 D_{\lambda}(y, t) = (\lambda (y-y_{0}), \lambda^2 (t-T)).
\end{equation}
Denote the new time parameter by $s$. Then $t = T + \frac{s}{\lambda^2}$. Let
\begin{equation*}
 M^{\lambda}_{s} \equiv M^{(y_{0}, T), \lambda}_{s} = D_{\lambda} (M_{t}) = \lambda (M_{T + \frac{s}{\lambda^2}} -y_{0}).
\end{equation*}
Then $(M^{\lambda}_{s})$ is a solution of the mean curvature flow in $B_{\lambda}(0)$ for $s\in [-\lambda^2 T, 0)$. 
Denote by $d\mu^{\lambda}_{s}$ the induced volume form on $M^{\lambda}_{s}$.
Let $\rho_{y_{0, T}}: \RR^{n+1}\times (-\infty, T) \to \RR$ be the backward heat kernel at $(y_{0}, T)$, i.e, 
\begin{equation}
 \rho_{y_{0}, T} (y, t) = \frac{1}{[4\pi(T-t)]^{n/2}} \mathrm{exp}(-\frac{\abs{y-y_{0}}^2}{4(T-t)}).
\label{BWH}
\end{equation}
The monotonicity formula of Huisken \cite{Huisken90} says that
\begin{equation}
\frac{d}{dt}\int_{M_{t}} \rho_{y_{0}, T} d\mu_{t} = -  \int_{M_{t}} \rho_{y_{0}, T} \abs{H + \frac{F^{\perp}}{2(T-t)}}^2d\mu_{t},
\label{Hmono}
\end{equation}
from which it follows that the limit $\lim_{t\rightarrow T} \int_{M_{t}} \rho_{y_{0}, T} d\mu_{t}$ exists. Here $F^{\perp}(\cdot,t)$ is the normal 
component of the position vector $F(\cdot,t)\in \RR^{n+1}$ in the normal space of $M_{t}$ in $\RR^{n+1}$. Via the parabolic dilation, (\ref{Hmono})
becomes
\begin{equation}
 \frac{d}{ds}\int_{M^{\lambda}_{s}} \rho_{0, 0} d\mu^{\lambda}_{s} = - 
 \int_{M^{\lambda}_{s}} \rho_{0, 0} \abs{H^{\lambda}_{s} - \frac{(F^{\lambda}_{s})^{\perp}}{2s}}^2d\mu^{\lambda}_{s}.
\label{DHmono}
\end{equation}
Fix $s_{0}<0$. Integrating both sides of (\ref{DHmono}) from $s_{0}-\tau$ to $s_{0}$ for $\tau>0$, we get
\begin{equation}
\int_{s_{0}-\tau}^{s_{0}}\int_{M^{\lambda}_{s}} \rho_{0, 0} \abs{H^{\lambda}_{s} - \frac{(F^{\lambda}_{s})^{\perp}}{2s}}^2d\mu^{\lambda}_{s} ds = 
\int_{M^{\lambda}_{s_{0}-\tau}} \rho_{0, 0} d\mu^{\lambda}_{s_{0}-\tau} - \int_{M^{\lambda}_{s_{0}}} \rho_{0, 0} d\mu^{\lambda}_{s_{0}}.
\label{Hint}
\end{equation}
Let $t_{1} = T + \frac{s_{0}}{\lambda^2}$. Then, by the invariance of $\int_{M_{t}} \rho_{y_{0}, T} d\mu_{t}$ under the parabolic scaling,
\begin{equation*}
 \int_{M_{t_{1}}} \rho_{y_{0}, T} d\mu_{t_{1}} = \int_{M^{\lambda}_{s_{0}}} \rho_{0, 0} d\mu^{\lambda}_{s_{0}}.
\end{equation*}
Letting $\lambda \rightarrow \infty$, one has $t_{1}\rightarrow  T$ and 
\begin{equation*}
 \lim_{\lambda\rightarrow \infty}\int_{M^{\lambda}_{s_{0}}} \rho_{0, 0} d\mu^{\lambda}_{s_{0}} = 
\lim_{t\rightarrow T}\int_{M_{t}} \rho_{y_{0}, T} d\mu_{t}.
\end{equation*}
Similarly,
\begin{equation*}
 \lim_{\lambda\rightarrow \infty}\int_{M^{\lambda}_{s_{0}-\tau}} \rho_{0, 0} d\mu^{\lambda}_{s_{0}-\tau} = 
\lim_{t\rightarrow T}\int_{M_{t}} \rho_{y_{0}, T} d\mu_{t}.
\end{equation*}
Therefore, by (\ref{Hint}),
\begin{equation}
\lim_{\lambda\rightarrow \infty} \int_{s_{0}-\tau}^{s_{0}}\int_{M^{\lambda}_{s}} \rho_{0, 0} \abs{H^{\lambda}_{s} - 
\frac{(F^{\lambda}_{s})^{\perp}}{2s}}^2d\mu^{\lambda}_{s} ds=0.
\label{vanish}
\end{equation}
On the other hand, the second fundamental form of $M^{\lambda}_{s}$ satisfies 
\begin{equation*}
 \mathrm{max}\abs{A}^2(\cdot,s) (M^{\lambda}_{s}) = \frac{1}{\lambda^2} \mathrm{max}\abs{A}^2(\cdot,t) (M_{t}) 
= -\frac{1}{s} (T-t) \mathrm{max}\abs{A}^2 (\cdot, t) (M_{t}).  
\end{equation*}
and thus, by (\ref{typeI}),
\begin{equation*}
 \mathrm{max}\abs{A}^2(\cdot,s) (M^{\lambda}_{s}) \leq  \frac{-C_{0}}{s}, ~\forall s\in
[-\lambda^2 T, 0).
\end{equation*}
In particular, for fixed $\delta \in (0, 1/2)$, the inequality
\begin{equation}
 \abs{A(y)}^2 \leq \frac{C_{0}}{\delta^2}
\end{equation}
holds for $y\in M^{\lambda}_{s}\cap B_{\lambda}$ and $s\in [-\lambda ^2 T, -\delta^2]$ and therefore for $y \in M^{\lambda}_{s}\cap B_{1/\delta}$ and 
$s\in [-1/\delta^2, -\delta^2]$ for $\lambda$ sufficiently large depending on $\delta$, say $\lambda \geq \lambda_{\delta}$. By the interior 
estimate \cite{EH}, one has for all $m\geq 0$
\begin{equation}
 \abs{\nabla ^{m} A(y)}^2 \leq \frac{C(C_{0}, m,n)}{\delta^{2(m +1)}}
\end{equation}
for $y \in M^{\lambda}_{s}\cap B_{1/2\delta}$ and $s \in [-1/4\delta^2, -\delta^2]$. Moreover, by (\ref{dist}),
\begin{equation*}
 \mathrm{dist} (0, M^{\lambda}_{s}) = \lambda\mathrm{dist} (y_{0}, M_{T + \frac{s}{\lambda^2}}) \leq \lambda \sqrt{2n(\frac{-s}{\lambda^2})} =
\sqrt{-2ns}
\end{equation*}
for the above times $s$ and $\lambda\geq \lambda_{\delta}$.
 By Arzela-Ascoli theorem combined with a diagonal sequence argument when letting $\delta\searrow 0$ for local graph representations of 
$(M^{\lambda}_{s})$, we can find a subsequence
$\lambda_{i}\rightarrow \infty$ such that $(M^{\lambda_{i}}_{s})$ converges smoothly on compact
subsets of $\RR^{n +1}\times (-\infty,0)$ to a smooth solution $(M^{\infty}_{s})_{s<0}$ of mean curvature flow. From
(\ref{vanish}), one sees that $H = \frac{1}{2s}F^{\perp}$ on $M^{\infty}_{s}$ for $s\in (s_{0}-\tau, s_{0})$. \\
Take $s_{0}\rightarrow 0$ and $\tau\rightarrow \infty$ to see that $H= \frac{1}{2s} F^{\perp}$ on $M^{\infty}_{s}$ for $-\infty<s<0$. In other words, 
$(M^{\infty})_{s}$ is a self-shrinking mean curvature flow. Moreover, one deduces from (\ref{Hbound}) and
$\abs{H^{\lambda}_{s}} =\abs{\frac{H_{t}}{\lambda}}$ that $H =0$ on $M_s^{\infty}$. Thus $M^{\infty}_{s}$ is a minimal cone for each $s<0$; see 
Corollary 2.8 in \cite{CM}. Because $M_{s}^{\infty}$ is smooth, it is a hyperplane.
Now, fix $s_{0}<0$. One has, as $i\rightarrow \infty$, $M^{\lambda_{i}}_{s_{0}}\rightarrow M^{\infty}_{s_{0}}\cong \RR^{n}$ and
$d\mu^{\lambda_{i}}_{s_{0}}\rightarrow d x^{n}.$ Thus
\begin{equation*}
 \lim_{i\rightarrow \infty }\int_{M^{\lambda_{i}}_{s_{0}}} \rho_{0, 0} d\mu^{\lambda_{i}}_{s_{0}} = \int_{M^{\infty}_{s_{0}}}\rho_{0,0} d x^n =1.
\end{equation*}
This implies that, for $t_{i} = T + \frac{s_{0}}{\lambda^2_{i}}$
\begin{equation}
\lim_{t_{i}\rightarrow T}\int_{M_{t_{i}}} \rho_{y_{0}, T} d\mu_{t_{i}} =1,
\end{equation}
and therefore 
$$\lim_{t\to T}\int_{M_t} \rho_{y_0,T}\, d\mu_t = 1.$$
By White's regularity theorem \cite{White}, the second fundamental form $\abs{A}(\cdot, t)$ of $M_{t}$ is bounded as $t\rightarrow T$
and $(y_{0}, T)$ is a regular point. Thus, the 
flow can be extended 
past time $T$.
\end{proof}
\begin{proof}[Proof of Theorem \ref{intMCbound}]
We will split the proof of Theorem \ref{intMCbound} in the following two lemmas.
\end{proof}

\begin{lemma}
Theorem \ref{intMCbound} holds for $\alpha > n+2$.
\end{lemma}

\begin{proof}
We use the same notations as in the proof of Theorem \ref{MCbound}. Note that under the parabolic 
dilations $D_{\lambda_{i}}$, the inequality (\ref{intHbound}) becomes
\begin{equation*}
 C_{0}^{\alpha} \geq \int_{0}^{T}\int_{M_{t}} \abs{H}^{\alpha} d\mu_{t} dt = 
\lambda_{i}^{\alpha - (n +2)}\int_{-\lambda_{i}^2 T}^{0}\int_{M^{\lambda_{i}}_{s}} \abs{H^{\lambda_{i}}_{s}}^{\alpha} d\mu^{\lambda_{i}}_{s} ds.
\end{equation*}
Thus
\begin{equation}
 \int_{-\lambda_{i}^2 T}^{0}\int_{M^{\lambda_{i}}_{s}} \abs{H^{\lambda_{i}}_{s}}^{\alpha} d\mu^{\lambda_{i}}_{s} ds 
\leq \frac{C_{0}^{\alpha}}{\lambda_{i}^{\alpha - (n +2)}}.
\end{equation}
Now, letting $i\rightarrow \infty$ as in the proof of Theorem \ref{MCbound}, we get a self-shrinking mean curvature flow $(M^{\infty})_{s}$
with the property that
\begin{equation}
 \int_{-\infty}^{0}\int_{M^{\infty}_{s}} \abs{H}^{\alpha} du^{\infty}_{s} ds =0
\end{equation}
because $\alpha > n +2$. Therefore $H =0$ on $M^{\infty}_{s}$. Now
we can argue similarly as in the proof of Theorem \ref{MCbound}. 
\end{proof}

\begin{lemma}
Theorem \ref{intMCbound} holds for $\alpha = n+2$.
\end{lemma}

\begin{proof}
We use the same notation as in the proof of Theorem \ref{MCbound}.
Under the parabolic dilations $D_{\lambda_{i}}$, the inequality (\ref{intHbound}) becomes
\begin{equation*}
 C_{0} \geq \int_{0}^{T}\int_{M_{t}} \abs{H}^{n+2} d\mu_{t} dt = 
\int_{-\lambda_{i}^2 T}^{0}\int_{M^{\lambda_{i}}_{s}} \abs{H^{\lambda_{i}}_{s}}^{n+2} d\mu^{\lambda_{i}}_{s} ds.
\end{equation*}
Letting $i\to\infty$ as before we get a complete and smooth self-shrinker $M^{\infty}_s$ in the limit
with the property that
\begin{equation}
\label{eq-finite}
\int_{-\infty}^{0}\int_{M^{\infty}_{s}} \abs{H}^{n+2} d\mu^{\infty}_{s} ds \le C_0 < \infty.
\end{equation}
Our self-shrinker satisfies
$$H = \frac{\langle x,\nu\rangle}{(-2s)},$$
which is equivalent to saying that $M_s = \sqrt{-s}M_{-1}$, where $M_s$ satisfies the mean
curvature flow. Notice that
\begin{eqnarray*}
\int_{M_s} \abs{H}^{n+2}(\cdot,s)d\mu_s &=& \int_{M_{-1}} \left (\frac{\abs{H}(\cdot,-1)}{\sqrt{-s}}\right)^{n+2} 
\cdot (-s)^{\frac{n}{2}}d\mu_{-1} \\
&=& \frac{1}{(-s)}\cdot \int_{M_{-1}}\abs{H}^{n+2}(\cdot,-1)\, d\mu_{-1} \\
&=& \frac{a}{(-s)},
\end{eqnarray*}
where $a := \int_{M_{-1}}\abs{H}^{n+2}(\cdot,-1)\, d\mu_{-1}$. If $a > 0$ then
$$\int_{-\infty}^0\int_{M^{\infty}_{s}} \abs{H}^{n+2} d\mu^{\infty}_{s} ds = a\cdot \int_{-\infty}^0\frac{ds}{(-s)} = \infty,$$
which contradicts (\ref{eq-finite}). Therefore $a = 0$, which implies $H(\cdot, -1) = 0$ on $M^{\infty}_{-1}$. 
Similar argument shows that $H(\cdot,s) = 0$ on $M^{\infty}_{s}$ for every $s < 0$. To prove that $(y_0,T)$
is a regular point of the flow we argue as in the proof of Theorem \ref{MCbound}.
\end{proof}

\section{Extension results for surfaces}
\label{3d}

In \cite{I1} Ilmanen proposed the following conjecture.

\begin{conj}[Multiplicity One Conjecture]
If $M^{2}_0$ is embedded in $\mathbb{R}^3$, then for any family of rescalings $\lambda_{j}(M_{\lambda_{j}^{-2} s + T} -y_{0})$ with
$\lambda_{j}\rightarrow \infty$, there is a subsequence smoothly converging and with multiplicity one 
to the blowup $N_t$, that is, there are no concentration points or multiple layers in the limit.
\label{MultOne}
\end{conj}

Theorem \ref{3dcase} assumes that the Conjecture above holds and its proof is given below.

\begin{proof}[Proof of Theorem \ref{3dcase}]
In this proof, $n=2$. Without loss of generality, assume that $M^{n}\subset B_{1}(0) \subset \RR^{n+1}$. Let $y_{0}\in \RR^{n+1}$ be a point 
reached by the mean curvature 
flow (\ref{MCF1}) at time $T$, that is, there exists a sequence $(y_{j}, t_{j})$ with $t_{j}\nearrow T$ so that $y_{j}\in M_{t_{j}}$ and 
$y_{j}\rightarrow y_{0}$. We show that $(y_{0}, T)$ is a regular point of (\ref{MCF1}). \\
As in the proof of Theorem \ref{MCbound}, let $t = T + \frac{s}{\lambda^2}$ and
\begin{equation*}
 M^{\lambda}_{s} \equiv M^{(y_{0}, T), \lambda}_{s} = \lambda (M_{T + \frac{s}{\lambda^2}} -y_{0}).
\end{equation*}
Then $(M^{\lambda}_{s})$ is a solution of the mean curvature flow in $B_{\lambda}(0)$ for $s\in [-\lambda^2 T, 0)$. For any set $A\subset \RR^{n +1}$, let 
us define the parabolically rescaled measures at $(y_{0}, T)$:
\begin{equation*}
 \mu_{s}^{\lambda} (A) = \lambda^{-n} \mathcal{H}^{n}\lfloor M_{s}^{\lambda} (\lambda\cdot A).
\end{equation*}
Let $\rho_{y_{0, T}}: \RR^{n+1}\times (-\infty, T) \to \RR$ be the backward heat kernel at $(y_{0}, T)$ as defined in (\ref{BWH}).
Then, a result on weak existence of blow ups of Ilmanen and White (see Lemma 8, page 14 of \cite{I1} and also \cite{WhiteCrelle}) 
says that: there exists a subsequence $\lambda_{j}$ and a limiting 
Brakke flow \cite{Brakke} $\{\nu_{s}\}_{s<0}$ (also known as a {\it tangent flow} ) such that $\mu_{s}^{\lambda_{j}}\rightharpoonup \nu_{s}$ in the sense of 
Radon measures for all $s<0$ and the following
statements hold:\\
(a) (self-similarity) $\nu_{s}(A) = \nu_{s}^{\lambda} (A) \equiv \lambda^{-n}\nu_{\lambda^2 s} (\lambda\cdot A)$, for all 
$s<0$ and for all $\lambda>0$\\
(b) (tangent flow is a self-shrinker) $\nu_{-1}$ satisfies
\begin{equation}
 \overrightarrow{H}(x) + \frac{S(x)^{\perp}\cdot x}{2} =0, ~\nu_{-1}~ \mathrm{a. e.}~ x
\end{equation}
(c) Furthermore, Huisken's integral converges
\begin{equation}
 \int \rho_{0,0} (x, -1) d\nu_{s}(x) = \lim_{t\nearrow T} \int \rho_{y_{0}, T} d\mu_{t}, s<0.
\label{Huiskenconv}
\end{equation}
Equivalently, a subsequence of rescaled solutions  $M_s^{\lambda}$ converges weakly to a limiting 
flow $X_s$ that is called a {\it tangent flow} at $(y_0,T)$. We know $X_s$ is a self shrinker. Ilmanen showed in \cite{I1} that it has to be smooth.  Our 
proof will rely on this fact and the validity of Multiplicity One Conjecture. Let us briefly explain the notations used in $(b)$.\\
For a locally $n$-rectifiable Radon measure $\mu$, we define its $n$-dimensional approximate tangent plane $T_{x}\mu$ (which exists $\mu$-a.e x) by
\begin{equation*}
 T_{x}\mu (A) = \lim_{\lambda \rightarrow 0} \lambda^{-n} \mu (x + \lambda\cdot A).
\end{equation*}
The tangent plane $T_{x}\mu$ is a positive multiple of $\mathcal{H}^{n}\lfloor P$ for some $n$-dimensional plane $P$. Let $S: \RR^{n +1} \longrightarrow G(n +1, n)$
denotes the $\mu-$ measurable function that maps $x$ to the geometric tangent plane, denoted by $P$ above. An important
quantity is the first variation of $\mu$, denoted by $\int div_{S(x)}X(x) d\mu(x)$ for $X\in C^{\infty}_{c}(\RR^{n +1}, \RR^{n +1})$.  Here $div_{S}X
= \sum_{i=1}^{n} D_{e_{i}}X.e_{i}$ where $e_{1}, \cdots, e_{n}$ is any orthonormal basis of $S$. We also denote by $S$
the orthogonal projection onto $S$ and thus $div_{S}X$ can be written as $S: DX$. Now, under suitable assumptions, we can 
define the generalized mean curvature vector $\overrightarrow{H} = \overrightarrow{H}_{\mu} \in L^{1}_{\mathrm{loc}}(\mu)$ of $\mu$ as follows
\begin{equation}
 \int div_{S}X d\mu =\int -\overrightarrow{H}\cdot X d\mu
\label{firstvar}
\end{equation}
for all $X\in C_{c}^{\infty}(\RR^{n +1}, \RR^{n +1})$. Note that when $\mu$ is the surface measure of a smooth $n$-dimensional manifold $M$, the generalized mean curvature vector $\overrightarrow{H}$ of 
$\mu$ is also the classical mean curvature vector of $M$; see Corollary 4. 3 in \cite{Schatzle1}.
From (\ref{firstvar}) and the definition of $\mu_{s}^{\lambda}$, one sees that the mean curvature vector 
$\overrightarrow{H}_{s}^{\lambda}$ of $\mu_{s}^{\lambda}$ 
is $\frac{\overrightarrow{H}_{t}}{\lambda}$ where $\overrightarrow{H}_{t}$ is the mean curvature vector of $M_{t}$ where $t = T + \frac{s}{\lambda^2}$.
The lower semicontinuity of $\int\abs{H} d\mu$ asserts that
\begin{equation*}
 \int\abs{\overrightarrow{H}_{s}} d\nu_{s}\leq \liminf_{\lambda\rightarrow \infty}\int \abs{\overrightarrow{H}_{s}^{\lambda}} d\mu_{s}^{\lambda} \leq 
\limsup_{\lambda\rightarrow \infty} \int \frac{C_{0}}{\lambda} d\mu_{s}^{\lambda} =0.
\end{equation*}
Thus $\overrightarrow{H}_{s} =0$ for all $s<0$. Now, because $X_s$ is smooth for all $s<0$, the weak mean curvature vector $\overrightarrow{H}_s$ coincides with the mean curvature vector in classical sense.  Thus we have a smooth solution $X_s$ that is a self-shrinker with $H = 0$ and therefore by the result in \cite{CM} it has to be a hyperplane. Furthermore
$\nu_{s}$ represents the surface measure of the plane $X_{s}$ with multiplicity one by the validity of the Multiplicity One
Conjecture. Using the convergence of Huisken's integral (\ref{Huiskenconv}), we see that
\begin{equation*}
 \lim_{t\nearrow T} \int \rho_{y_{0}, T} d\mu_{t} =1.
\end{equation*}
 By White's regularity theorem \cite{White}, the second fundamental form $\abs{A}(\cdot, t)$ of $M_{t}$ is bounded as $t\rightarrow T$
and $(y_{0}, T)$ is a regular point. Thus, the 
flow can be extended 
past time $T$.
\end{proof}

We conclude this section by the proof of Theorem \ref{3dreg}, which can be viewed as a local regularity result without a smallness condition.

\begin{proof}[Proof of Theorem \ref{3dreg}]
We will use the same notation as in the proof of Theorem \ref{3dcase}. Note that, when $n=2$, by the fact that $\int H_s^2$ is bounded 
(follows from the Gauss-Bonnet theorem for surfaces) and Allard's Compactness Theorem \cite{S},  each 
Radon measure $\nu_{s}$ is
integer $2$-rectifiable, that is
\begin{equation*}
 d\nu_{s} = \theta_{s}(x)d\mathcal{H}^{2}\lfloor X_{s}
\end{equation*}
where $X_{s}$ is an $\mathcal{H}^{2}$-measurable, $2$-rectifiable set and $\theta_{s}$ is an $\mathcal{H}^{2}\lfloor X_{s}$-integrable, integer valued "multiplicity function". 

Furthermore
the mean curvature vector $\overrightarrow{H}_{s}$ of $\nu_{s}$ satisfies $\overrightarrow{H}_{s}\in L^{\infty}(\nu_{s})$. {\it Here is the only 
place we wish 
to use (\ref{Hbound3d}). The same argument as in the proof of Theorem \ref{3dcase} implies $X_s$ is a plane. }\\
 \h The key point of our proof is the following Constancy Theorem due to Sch\"{a}tzle.
\begin{theorem} (Sch\"{a}tzle's Constancy Theorem)
Let $\mu=\theta \mathcal{H}^{n}\lfloor M$ be an integral $n$-varifold in the open set $\Omega\subset R^{n+m}$, $M\subset \Omega$ a 
connected $C^{1}$-n-manifold, $\theta: M\rightarrow N_{0}$ be $\mathcal{H}^{n}$-measurable with weak 
mean curvature $\overrightarrow{H}_{\mu}\in L_{loc}^{1}(\mu)$, that is
\begin{equation}
 \int \mathrm{div}_{\mu} \eta d \mu = \int_{M} \mathrm{div}_{M}\eta \theta d\mathcal{H}^{n} = -\int <\overrightarrow {H}_{\mu}, \eta> d\mu~~
\forall \eta\in C^{1}_{0}(\Omega, \RR^{n+ m}).
\label{weakMC}
\end{equation}
Then $\theta$ is a constant: $\theta \equiv\theta_{0}\in N_{0}$. Here $N_{0}$ is the set of all nonnegative integers and $<\cdot>$ is the 
standard Euclidean inner product on $\RR^{n + m}.$
\label{Schatzle}
 \end{theorem}
Now $\theta_{s}$ is a constant and $X_{s}$ is a plane. Thus by the convergence of Huisken's integral (\ref{Huiskenconv}), we see that
\begin{equation*}
 \lim_{t\nearrow T} \int \rho_{y_{0}, T} d\mu_{t} = \int \rho_{0,0} (x, -1) d\nu_{s}(x) =  \int \rho_{0,0} (x, -1) \theta_{s} d\mathcal{H}^{2}\lfloor X_{s}
=\theta_{s}.
\end{equation*}
By (\ref{below2}) and Proposition $2.10$ in \cite{White}, $1 \le \theta_{s}<2$. It follows from the integrality of $\theta_{s}$ that $\theta_{s}\equiv 1$. Now, our result follows
from White's local regularity theorem \cite{White}.
\end{proof}
\begin{proof}[Proof of Theorem \ref{Schatzle}] The proof of this theorem can be found in \cite{Le}, Theorem 4.1. We include here for the reader's
convenience. We consider locally $C^{1}$-vector fields $\nu^{1}, \cdots, \nu^{m}$ on $M$, which are an orthonormal basis of 
the orthogonal complement $TM^{\perp}$ of the tangent bundle $TM$ in $T\RR^{n +m}$. For $x\in M$, we 
choose an orthonormal basis $\tau_{1},\cdots, \tau_{n}$ of the tangent space $T_{x}M$ of $M$ at $x$. We decompose 
$\eta\in C^{1}_{0}(\Omega, \RR^{n + m})$ into $
 \eta =\eta^{tan} + \eta^{\perp}, $
where $$\eta^{tan}(x)=\pi_{T_{x}M}(\eta(x))\in T_{x}M,~\h
\eta^{\perp}(x)=\pi_{T_{x}M^{\perp}}(\eta(x))=\sum_{j=1}^{m}<\nu^{j}, \eta(x)>\nu^{j}\in T_{x}M^{\perp}.
 $$
Here, we have denoted $\pi_{V}$ the orthogonal projection operator on the subspace $V$ of $\RR^{n +m}$.
In particular, $
 \eta^{tan}, \eta^{\perp}\in C^{1}_{0}(\Omega).$ Then, we have $
 \mathrm{div}_{M} \eta = \mathrm{div}_{M} \eta^{tan}  + \mathrm{div}_{M} \eta^{\perp}. $
Let $D$ be the standard differentiation operator on $\RR^{n +m}$ 
and $A_{M}$ the second fundamental form of $M$. Denote by $\overrightarrow{H}_{M}$ the weak mean curvature of $M$. Then
\begin{equation*}
\overrightarrow{H}_{M} = \sum_{i=1}^{n} A_{M}(\tau_{i}, \tau_{i}).
\end{equation*}
We have
\begin{eqnarray*}
 \mathrm{div}_{M} \eta^{\perp} &=& \sum_{i=1}^{n}<\tau_{i}, \nabla^{M}_{\tau_{i}}\eta^{\perp}> =\sum_{i=1}^{n}\sum_{j=1}^{m}
<\tau_{i}, D_{\tau_{i}}\left(<\nu^{j}, \eta(x)>\nu^{j}\right)> \\&=&
\sum_{i=1}^{n}\sum_{j=1}^{m}
<\nu^{j},\eta>< \tau_{i}, D_{\tau_{i}}\nu^{j}>
= -<\eta, \sum_{i=1}^{n} A_{M}(\tau_{i}, \tau_{i})>
= -<\eta, \overrightarrow{H}_{M}>.
\end{eqnarray*}
From (\ref{weakMC}), we can calculate
\begin{eqnarray*}
 -\int <\overrightarrow{H}_{\mu},\eta> d\mu = -\int_{M}<\overrightarrow{H}_{\mu},\eta> \theta d\mathcal{H}^{n}
 &=& \int_{M}\mathrm{div}_{M}\eta\theta d\mathcal{H}^{n}\\
&=& \int_{M}\mathrm{div}_{M}\eta^{tan}\theta d\mathcal{H}^{n} +  \int_{M}\mathrm{div}_{M}\eta^{\perp}\theta d\mathcal{H}^{n}\\
&=&  \int_{M}\mathrm{div}_{M}\eta^{tan}\theta d\mathcal{H}^{n} - \int_{M}<\overrightarrow{H}_{M},\eta> \theta d\mathcal{H}^{n}.
\end{eqnarray*}
Let us make some special choices of $\eta$. First, for $\eta =\eta^{\perp} \in TM^{\perp}$, we conclude 
that the projection $\overrightarrow{H}^{\perp}_{\mu}$ of $\overrightarrow{H}_{\mu} $ on 
$TM^{\perp}$ satisfies $\overrightarrow{H}^{\perp}_{\mu} = \overrightarrow{H}_{M}$. Since $\mu$ is integral, 
we get $\overrightarrow{H}_{\mu} \bot T\mu = TM$ by Theorem 5. 8 in Brakke \cite{Brakke} and conclude $
 \overrightarrow{H}_{\mu} = \overrightarrow{H}_{M}.$
Finally, if we choose $\eta$ such that $\eta = \eta^{tan}\in TM$ then
\begin{equation*}
 \int_{M}\mathrm{div}_{M}\eta^{tan}\theta d\mathcal{H}^{n} =0.
\end{equation*}
Calculating in local coordinates, this yields $\nabla_{M}\theta =0$ weakly. Hence $\theta \equiv \theta_{0}$ is constant, as $M$ 
is connected.
\end{proof}

\section{Some global results on the extension of (\ref{MCF1})}
\label{sec-global}

In this section we give  global conditions for extending a smooth solution to (\ref{MCF1}), which has been a subject of study in \cite{LS}.

\begin{proof}[Proof of Theorem \ref{SFbound}]
We argue by contradiction. Suppose that $T$ is the extinction time of the flow. Then, 
by Theorem \ref{Aunbound}, $\abs{A}$ is unbounded. Therefore, there exists a sequence of points $(x_{i}, t_{i})$ with $x_{i}\in M_{t_{i}}$ 
such that
\begin{equation}
 Q_{i}: = \abs{A}(x_{i}, t_{i}) = \max_{0\leq t\leq t_{i}} \max_{x\in M_{t}} \abs{A}(x, t) \to +\infty.
\label{maxMC}
\end{equation}
Consider the sequence $\tilde{M}^{i}_{t}$ of rescaled solutions for $t\in [0,1]$ defined by
\begin{equation*}
 \tilde{F}_{i} (\cdot,t) = Q_{i} (F(\cdot, t_{i} + \frac{t-1}{Q^2_{i}}) - x_i).
\end{equation*}
The sequence of rescaled solutions $\tilde{M}_t^i$ converges (see \cite{ChenHe}) to a complete smooth solution to the mean curvature flow, call 
it $\tilde{M}_t$ for $t\in [0,1]$ with the property that 
\begin{equation}
\label{eq-nonzero}
|\tilde{A}|(0,1) = 1.
\end{equation}
If $g$ and $A:= \{h_{jk}\}$ are the induced metric, the mean curvature and the second fundamental form of $M_{t}$, respectively, then
the corresponding rescaled quantities are given by
\begin{equation*}
 \tilde{g}_{i} = Q^2_{i} g;~ \abs{\tilde{A}_{i}}^2 =\frac{\abs{A}^2}{Q^2_{i}}.
\end{equation*}
We calculate
\begin{multline}
\lim_{i\rightarrow\infty } \left\{ \int^{1}_{0}\left(\int_{(\tilde{M_{i}})_{t}\cap B(0,1)}\abs{\tilde{A_{i}}}^{q} d\mu\right)^{\frac{p}{q}} dt\right\}^{\frac{1}{p}} \\=\lim_{i\rightarrow\infty }  \left\{\int^{t_{i}}_{t_{i}-\frac{1}{Q^2_{i}}}\left(\int_{M_{t}\cap B(0,\frac{1}{Q_{i}})}\abs{A}^{q} d\mu\right)^{\frac{p}{q}} dt\right\}^{\frac{1}{p}}
 \leq \lim_{i\rightarrow\infty }  \left\{\int^{t_{i}}_{t_{i}-\frac{1}{Q^2_{i}}}
\left(\int_{M_{t}}\abs{A}^{q} d\mu\right)^{\frac{p}{q}} dt\right\}^{\frac{1}{p}}
= 0.
\label{vanishA}
\end{multline}
The last step follows from the facts that
\begin{equation*}
  \left(\int_{0}^{T}\left(\int_{M_{t}}\abs{A}^{q} d\mu\right)^{p/q} dt\right)^{1/p}< \infty; \,\,\,\, \lim_{i\rightarrow \infty} \frac{1}{Q^2_{i}} =0.
\end{equation*}
By Fatou's lemma and (\ref{vanishA}) it follows that
$$\int_0^1\left(\int_{\tilde{M}_t\cap B(0,1)} |\tilde{A}|^q\right)^{\frac{p}{q}} = 0.$$
By the smoothness of $\tilde{M}_t$, this implies $|\tilde{A}|(x,t) \equiv  0$ for all $x\in \tilde{M}_t\cap B(0,1)$ and all $t\in [0,1]$.
This contradicts (\ref{eq-nonzero}).
\end{proof}

\section{Some local regularity results and applications}
\label{sec-local}
   
\subsection{$\e$-regularity theorem for the mean curvature flow}
\label{sec-ep}

In this section we prove Theorem \ref{CS}, which is  parabolic version of the epsilon regularity theorem for minimal surfaces proven by Choi and  Schoen \cite{CS}. A version of the epsilon regularity theorem for mean curvature flow has been obtained by Ecker in \cite{EckerCalvar}. He required smallness of the supremum over small time intervals of spatial $L^n$ norms of $|A|$ over small balls. 

\begin{proof}[Proof of Theorem \ref{CS}]
We may assume without loss of generality that the flow $\mathcal{M} = (M_{t})_{t<t_{0}}$ is smooth up to and including time $t_{0}$, because we can
first prove the theorem with $t_{0}$ replaced by $t_{0}-\alpha$ for fixed $\alpha>0$ and then (since the right hand side of our desired
inequality is independent of $\alpha>0$) let $\alpha \searrow 0$ afterwards. \\
\h Let 
\begin{equation*}
 F(\delta) = \sup_{t\in [t_{0}-(\sigma-\delta)^2,t_{0}]} \sup_{x\in B(x_{0}, \sigma-\delta)\cap M_{t}} 
\delta^2 \abs{A}^2 (x,t). 
\end{equation*}
Since our flow is smooth up to time $t_{0}$, $F(0) =0$. Thus, there exists $\delta_{\ast}\in (0, \sigma/2]$ such that
$F(\delta_{\ast}) =\max_{0\leq \delta\leq \sigma/2} F(\delta)$. It suffices to show that
\begin{equation}
 F(\delta_{\ast})< \e_{0}^{\frac{-2}{n +2}} 
(\int_{t_{0}-\sigma^2}^{t_{0}} \int_{M_{t}\cap B(x_{0},\sigma)}\abs{A}^{n +2} d\mu dt)^{\frac{2}{n +2}} \equiv (\e_{0}^{-1}\eta)^{\frac{2}{n +2}}.
\end{equation}
Suppose not, then 
\begin{equation}F(\delta_{\ast}) \geq (\e_{0}^{-1}\eta)^{\frac{2}{n +2}}.
 \label{Fbig}
\end{equation}
 Because the flow is defined up to and including time $t_{0}$, we
can find $t_{\ast}\in [t_{0}-(\sigma-\delta_{\ast})^2, t_{0}]$ and $x_{\ast}\in \overline{B}(x_{0}, \sigma-\delta_{\ast})\cap M_{t_{0}}$ such that
\begin{equation}
 \delta_{\ast}^{2} \abs{A}^{2}(x_{\ast}, t_{\ast}) = F(\delta_{\ast}).
\label{goodradius}
\end{equation}
It follows from $\delta_{\ast}\in (0, \sigma/2]$ that
\begin{equation}
 B(x_{\ast}, \frac{\delta_{\ast}}{2})\subset B(x_{0}, \sigma-\frac{\delta_{\ast}}{2});~
 [t_{\ast} -\frac{\delta_{\ast}^2}{4}, t_{\ast}] \subset [t_{0}- (\sigma-\frac{\delta_{\ast}}{2})^2, t_{0}].
\label{shrinking}
\end{equation}
By the choice of $\delta_{\ast}, t_{\ast}$ and $x_{\ast}$,
\begin{equation*}
 (\frac{\delta_{\ast}}{2})^2 \sup_{t\in [t_{0}- (\sigma-\frac{\delta_{\ast}}{2})^2, t_{0}]}\sup_{x\in B(x_{0}, \sigma-\frac{\delta_{\ast}}{2})
\cap M_{t}}
\abs{A}^2(x,t) \leq F(\delta_{\ast}) = \delta_{\ast}^2\abs{A}^2(x_{\ast}).
\end{equation*}
Hence
\begin{equation*}
 \sup_{t\in [t_{0}- (\sigma-\frac{\delta_{\ast}}{2})^2, t_{0}]}\sup_{x\in B(x_{0}, \sigma-\frac{\delta_{\ast}}{2})\cap M_{t}}
\abs{A}^2(x,t) \leq 4\abs{A}^2(x_{\ast}, t_{\ast})
\end{equation*}
and thus, it follows from (\ref{shrinking}) that
\begin{equation}
 \sup_{t\in  [t_{\ast} -\frac{\delta_{\ast}^2}{4}, t_{\ast}]}\sup_{x\in  B(x_{\ast}, \frac{\delta_{\ast}}{2})\cap M_{t}}
\abs{A}^2(x,t) \leq 4\abs{A}^2(x_{\ast}, t_{\ast}).
\label{inside}
\end{equation}
We now rescale our mean curvature flow by setting
\begin{equation*}
 \tilde{F}(\cdot, t) = Q F(\cdot, t_{\ast} + \frac{t-1}{Q^2}), ~\tilde{M}_{t} = \tilde{F}(M^{n}, t)
\end{equation*}
where
\begin{equation*}
 Q = 2 (\e_{0}\eta^{-1})^{\frac{1}{n +2}} \abs{A}(x_{\ast}, t_{\ast}).
\end{equation*}
Then we have a mean curvature flow $\tilde{M}_{t}$ on $B(x_{\ast}, Q\delta_{\ast}/2)$ for $t\in [0,1]$.
Let $\tilde{g} = Q^2 g$ be the induced metric on $\tilde{M}_{t}$ and let $\tilde{B}(x_{\ast}, r)$ be the geodesic ball w. r. t. the metric
$\tilde{g}$ and centered at $x_{\ast}$ with 
radius $r$.
By (\ref{Fbig}) and (\ref{goodradius}), we have
\begin{equation}
 \frac{Q\delta_{\ast}}{2}\geq 1.
\end{equation}
This combined with (\ref{inside}) gives
\begin{equation}
 \sup_{t\in [0,1]}\sup_{x\in \tilde{B}(x_{\ast}, 1)\cap \tilde{M}_{t}} \abs{\tilde{A}}^2(x,t) \leq 4 \abs{\tilde{A}}^2(x_{\ast},1)\leq 1.
\label{boundA}
\end{equation}
Note that the last inequality follows from the facts that 
\begin{equation*}
 4\abs{\tilde{A}}^{2}(x_{\ast},1) = \frac{4\abs{A}^2(x_{\ast}, t_{\ast})}{Q^2} = (\e_{0}^{-1}\eta)^{\frac{2}{n +2}} \leq 1.
\end{equation*}
To obtain a contradiction, we will use the inequality
\begin{equation}
 (\partial_{t}-\Delta) \abs{\tilde{A}}^2 \leq 2 \abs{\tilde{A}}^{4}.
\label{Simons}
\end{equation}
This is a differential inequality of the form $ (\partial_{t}-\Delta) v \leq fv$ where $v=\abs{\tilde{A}}^2$ and $f=2\abs{\tilde{A}}^2$. Furthermore
$f$ satisfies a smallness condition:
\begin{equation*}
 \int_{0}^{1}\int_{\tilde{M}_{t}\cap \tilde{B}(x_{\ast}, 1)} f^{\frac{n +2}{2}} d\tilde{\mu} dt \leq 2^{n +2}\eta \leq 2^{n +2} \e_{0}.
\end{equation*}
Thus, we can localize our estimates in Lemmas 5. 1 and 6.1 in \cite{LS} to obtain the following inequality
\begin{equation}
 \abs{\tilde{A}}^2(x_{\ast},1)\leq \sup_{t\in [1/12,1]}\sup_{x\in \tilde{B}(x_{\ast},1/2)} \abs{\tilde{A}}^2(x,t)\leq C(n) 
(\int_{0}^{1}\int_{\tilde{M}_{t}\cap \tilde{B}(x_{\ast}, 1)} \abs{\tilde{A}}^{n +2} d\tilde{\mu} dt)^{\frac{2}{n +2}} \leq C(n) \eta^{\frac{2}{n +2}}.
 \label{keycontra}
\end{equation}
There is a simple proof of this inequality. It goes as follows. 
Note that for $t\in [0,1]$ and $x\in \tilde{B}(x_{\ast},1)$, (\ref{boundA}) and (\ref{Simons}) give
\begin{equation*}
 (\partial_{t}-\Delta) \abs{\tilde{A}}^2 \leq 2 \abs{\tilde{A}}^{2}
\end{equation*}
or equivalently $ (\partial_{t}-\Delta) (e^{-2t}\abs{\tilde{A}}^2)\leq 0.$
Now, we can apply Moser's mean value inequality (\cite{EckerCalvar}, Proposition 1.6) for $e^{-2t}\abs{\tilde{A}}^2$ to obtain a 
constant $C_{1}(n)$ depending only on $n$ such that 
\begin{equation}
 \abs{\tilde{A}}^2(x_{\ast},1) \leq C_{1}(n) \int_{0}^{1}\int_{\tilde{M}_{t}\cap \tilde{B}(x_{\ast}, 1)} \abs{\tilde{A}}^2 d\tilde{\mu} dt.
\end{equation}
By H\"{o}lder inequality
\begin{equation}
  \abs{\tilde{A}}^2(x_{\ast},1) \leq C_{1}(n) (\int_{0}^{1}\int_{\tilde{M}_{t}\cap \tilde{B}(x_{\ast}, 1)} d\tilde{\mu} dt)^{\frac{n}{n +2}}
(\int_{0}^{1}\int_{\tilde{M}_{t}\cap \tilde{B}(x_{\ast}, 1)} \abs{\tilde{A}}^{n +2} d\tilde{\mu} dt)^{\frac{2}{n +2}}.
\end{equation}
By (\ref{boundA}) and the Gauss equation 
\begin{equation*}
 \tilde{R}_{ik} = \tilde{H}\tilde{h}_{ik} - \tilde{h}_{il}\tilde{g}^{lj}\tilde{h}_{jk}
\end{equation*}
one easily sees that the Ricci tensor satisfies $\tilde{R}_{ik}\geq - (n-1).$ By the Bishop-Gromov volume comparison theorem, for each time $t\in [0,1]$, one has
$\int_{\tilde{M}_{t}\cap \tilde{B}(x_{\ast},1)}d\tilde{\mu}
\leq V(n)$ where $V(n)$ denotes the volume of a unit geodesic ball in an n-dimensional space form of 
constant curvature $-1$. Thus
\begin{multline}
 \abs{\tilde{A}}^{2}(x_{\ast},1)\leq C_{1}(n)V(n)^{\frac{n}{n +2}}(\int_{0}^{1}\int_{\tilde{M}_{t}\cap \tilde{B}(x_{\ast}, 1)} \abs{\tilde{A}}^{n +2}
 d\tilde{\mu} dt
)^{\frac{2}{n +2}} \\= C_{1}(n)V(n)^{\frac{n}{n +2}}(\int_{t_{\ast}-\frac{1}{Q^2}}^{t_{\ast}}\int_{M_{t}\cap B(x_{\ast}, \frac{1}{Q})}
\abs{A}^{n +2} d\mu dt)^{\frac{2}{n +2}}\\
\leq C_{1}(n)V(n)^{\frac{n}{n +2}} (\int_{t_{\ast}-\delta_{\ast}^2/2}^{t_{\ast}}\int_{M_{t}\cap B(x_{\ast}, 
\delta_{\ast}/2)}\abs{A}^{n +2} d\mu dt)^{\frac{2}{n +2}}\\
\leq C_{1}(n)V(n)^{\frac{n}{n +2}}
 (\int_{t_{0}-\sigma^2}^{t_{0}} \int_{M_{t}\cap B(x_{0},\sigma)}\abs{A}^{n +2} d\mu dt)^{\frac{2}{n +2}} = C_{1}(n)V(n)^{\frac{n}{n +2}} 
\mu^{\frac{2}{n +2}}.
\end{multline}
Consequently,
\begin{equation*}
 C_{1}(n)V(n)^{\frac{n}{n +2}}(\mu)^{\frac{2}{n +2}} \geq  \abs{\tilde{A}}^{2}(x_{\ast},1) = \frac{1}{4} (\e_{0}^{-1}\mu)^{\frac{2}{n +2}}.
\end{equation*}
This is a contradiction if $\e_{0}$ is small.
\end{proof}
\begin{remark}
\label{lemmapq}
In general, we can modify the proof of Theorem \ref{CS} to obtain the following result.
 Suppose $\mathcal{M} = (M_{t})$ is a smooth, properly embedded
solution of the mean curvature flow in $B(x_{0}, \rho)\times (t_{0}-\rho^2, t_{0})$ which reaches $x_{0}$ at time $t_{0}$. Let $p$ and $q$
be positive numbers satisfying
\begin{equation*}
 \frac{n}{q} + \frac{2}{p} =1.
\end{equation*}
Then, there exists $\e_{0} = \e_{0}(M_{0}, p, q)>0$ such that if $0<\sigma\leq \rho$ and
\begin{equation}
 \int_{t_{0}-\sigma^2}^{t_{0}} \left(\int_{M_{t}\cap B(x_{0},\sigma)}\abs{A}^{q} d\mu\right)^{p/q} dt <\e_{0}
\end{equation}
then
\begin{equation}
\label{eq-pq}
 \max_{0\leq\delta\leq \sigma/2} \sup_{t\in [t_{0}-(\sigma-\delta)^2,t_{0})} \sup_{x\in B(x_{0}, \sigma-\delta)\cap M_{t}} 
\delta^2 \abs{A}^2 (x,t) < \e_{0}^{\frac{-2}{p}} 
(\int_{t_{0}-\sigma^2}^{t_{0}}\left( \int_{M_{t}\cap B(x_{0},\sigma)}\abs{A}^{q} d\mu\right)^{p/q} dt)^{\frac{2}{p}}.
\end{equation}
\end{remark}
Theorem \ref{CS} can be extended to the case when an  ambient manifold is an arbitrary Riemannian manifold.

\begin{cor}
\label{cor-ext}
Let $n \ge 2$ and $N^{n+1}$ be a smooth complete, locally symmetric Riemannian manifold with bounded geometry. Let $M_0$ be a 
compact connected hypersurface without boundary which is smoothly immersed in $B(x_0,\rho) \subset N$. 
Suppose that $\mathcal{M} = (M_{t})$ is a smooth, properly embedded
solution of the mean curvature flow in $B(x_{0}, \rho)\times (t_{0}-\rho^2, t_{0})$ which reaches $x_{0}$ at time $t_{0}$. There exists $\epsilon_0 = \epsilon_0(M_0, N)$ such that if  $0<\sigma\leq \rho$ and
\begin{equation}
 \int_{t_{0}-\sigma^2}^{t_{0}} \int_{M_{t}\cap B(x_{0},\sigma)}\abs{A}^{n +2} d\mu dt <\e_{0}
\end{equation}
then
\begin{equation}
 \max_{0\leq\delta\leq \sigma/2} \sup_{t\in [t_{0}-(\sigma-\delta)^2,t_{0})} \sup_{x\in B(x_{0}, \sigma-\delta)\cap M_{t}} 
\delta^2 \abs{A}^2 (x,t) < \e_{0}^{\frac{-2}{n +2}} 
(\int_{t_{0}-\sigma^2}^{t_{0}} \int_{M_{t}\cap B(x_{0},\sigma)}\abs{A}^{n +2} d\mu dt)^{\frac{2}{n +2}}.
\end{equation}
\end{cor}

\begin{proof}
In the formulas that follow, if we mean the metric or the connection on $N$, this will be indicated by 
a bar, for example $\bar{g}_{\alpha\beta}$, etc. The Riemann curvature 
tensors on $M$ and $N$ will be denoted by $\rem = \{R_{ijkl}\}$ and $\overline{\rem} = \{\bar{R}_{\alpha\beta\gamma\delta}\}$. 
Let $\nu$ be the outer unit normal to $M_{t}$. For a fixed time $t$, we choose a local field of frame
$e_{0}, e_{1}, \cdots, e_{n}$ in $N$ such that when restricted to $M_{t}$, we have
$e_{0} = \nu, e_{i} =\frac{\partial F}{\partial x_{i}}$. The relations between $A = (h_{ij}), \rem$ and $\bar{\rem}$ are given by the equations of Gauss and Codazzi:
\begin{equation}
\label{eq-cod1}
R_{ijkl} = \bar{R}_{ijkl} + h_{ik}h_{jl} - h_{il}h_{jk},
\end{equation}
\begin{equation}
\label{eq-cod2}
\nabla_k h_{ij} - \nabla_j h_{ik} = \bar{R}_{0ijk}.
\end{equation}
Observe that we have the following evolution equation:
\begin{eqnarray*}
\frac{\partial}{\partial t}|A|^2 &=& \Delta |A|^2 - 2|\nabla A|^2 + 2|A|^2(|A|^2 + \overline{\ric}(\nu,\nu)) \nonumber \\
&-& 4(h^{ij}h_j^m\bar{R}_{mli}^l - h_{ij}h^{lm}\bar{R}_{milj}) - 2h^{ij}(\bar{\nabla}_j\bar{R}_{0li}^l + \bar{\nabla}_l\bar{R}_{0ij}^l),
\end{eqnarray*}
whose derivation can be found in \cite{Hu85}.
Since $N$ is, by our assumption,  locally symmetric, we have $\bar{\nabla}\overline{\rem} = 0$ and therefore the previous equation just reads as
\begin{eqnarray}
\label{eq-ev-rim}
\frac{\partial}{\partial t}|A|^2 &=& \Delta |A|^2 - 2|\nabla A|^2 + 2|A|^2(|A|^2 + \overline{\ric}(\nu,\nu)) \\
&-& 4(h^{ij}h_j^m\bar{R}_{mli}^l - h_{ij}h^{lm}\bar{R}_{milj}) \nonumber.
\end{eqnarray}
Using the evolution equation (\ref{eq-ev-rim}) and the bounds on the geometry of $N$  we obtain
$$\frac{\partial}{\partial t}|A|^2 \le \Delta |A|^2 + 2|A|^4 + C|A|^2.$$
After rescaling our solution and using (\ref{boundA}), we obtain
$$(\frac{\partial}{\partial t} - \Delta)|\tilde{A}|^2 \le C|\tilde{A}|^2.$$ 
If $f := e^{-C \cdot t}|\tilde{A}|^2$ then we have
$$(\frac{\partial}{\partial t} - \Delta) f \le 0.$$
If we take a trace of (\ref{eq-cod1}) in $jl$ we obtain
$$R_{ik} = g^{jl}\bar{R}_{ijkl} + h_{ik} H  - h_{il}h_{jk}g^{jl} \ge -C.$$ 
Applying the Moser mean value inequality to $f$ will lead to a contradiction in the same way as in the proof of Theorem \ref{CS}.
\end{proof}

\subsection{Some applications of Theorem \ref{CS}}
\label{sec-apply}
In this section, we give three applications of the local regularity results obtained in section \ref{sec-ep}. \\
\h The first application is a simple consequence of the Remark \ref{lemmapq}.  It gives a sufficient integral condition for  (\ref{MCF1}) to 
have a type I singularity. This will be achieved by showing that any type-I 
control on the $L^{s}$-norm ($s>n$) of the second fundamental form for all time slice
$t$ gives a type -I control on the second fundamental form. Precisely, we prove the following. 
 \begin{cor}
Let $s\in (n, \infty)$. Suppose there is a constant $C_{s}>0$ such that for any $T/2 \leq t<T$, we have
\begin{equation}
 \norm{A}_{L^{s}(M_{t})} \leq \frac{C_{s}}{(T-t)^{\frac{s-n}{2s}}}.
\label{Lsbound}
\end{equation}
Then (\ref{typeI}) holds. 
\label{improves}
\end{cor}

\begin{remark}
We say that the $L^s$-control on the second fundamental form given by (\ref{Lsbound}) is of type I. Notice that for the shrinking spheres $S^n$ we have the equality in (\ref{Lsbound}).
\end{remark}
\begin{proof}[Proof of Corollary \ref{improves}]
For $q = s > n$ there exists a positive number $p$ such that 
\begin{equation*}
 \frac{n}{s} + \frac{2}{p} =1.
\end{equation*}
Let $(x_{0}, t_{0})$ be arbitrary, where $0<t_{0}<T$.  Let $\sigma \in (0, \sqrt{t_{0}})$. Then, for any $t\in (t_{0}-\sigma^2, t_{0})$ we have 
\begin{equation}
 \int_{t_{0}-\sigma^2}^{t_{0}}\left(\int_{B(x_{0}, \sigma)\cap M_{t}}\abs{A}^{s} \right)^{\frac{2}{s-n}} \leq  \int_{t_{0}-\sigma^2}^{t_{0}} \frac{dt}{T-t} \le \frac{C\sigma^2}{T - t_0} =: C\alpha,
\end{equation}
where $\alpha(T - t_0) = \sigma^2$.
Fix $\alpha$ sufficiently small so that $ C \alpha \leq \e_{0}(M_{0}, p, q)$ where $\e_{0}(M_{0}, p, q)$ is the small constant in Remark \ref{lemmapq}. For this choice of $\alpha$,
the estimate (\ref{eq-pq}), taking $\delta = \frac{\sigma}{2}$ gives
\begin{equation*}
 \abs{A}^2 (x_{0}, t_{0}) \leq \frac{4}{\sigma^2} = \frac{C}{T-t_{0}},
\end{equation*}
and this completes the proof of our corollary.
\end{proof}
The second application is a lower bound on the $L^{s}$-norm ($s>n$) of the second fundamental form at each time slice. This lower bound can be viewed
as a slight generalization of Husiken's estimate \cite{Huisken90} where the case $s=\infty$ was considered. Let $s\in (n, \infty)$. Suppose that $T$ is the first singular time of the mean curvature flow. 
We are interested in the following question: \\
\h Does there exist a constant $C^{'}_{s}>0$ such that for any $t<T$, we have
\begin{equation}
 \norm{A}_{L^{s}(M_{t})} \geq \frac{C^{'}_{s}}{(T-t)^{\frac{s-n}{2s}}}?
\label{HuiskenLs}
\end{equation}
We prove a weaker version of the above inequality as follows
\begin{cor}
 For $t<T$, let $
 f(t) =\sup_{t_{1}\leq t} \norm{A}_{L^{s}(M_{t_{1}})}. $
Then there exists a constant $C^{'}_{s}>0$ such that
\begin{equation}
f(t) \geq \frac{C^{'}_{s}}{(T-t)^{\frac{s-n}{2s}}}
\end{equation}

\label{Lslemma}
\end{cor}
\begin{proof}[Proof of Corollary \ref{Lslemma}] 
For $q = s > n$ there exists a positive number $p$ such that 
\begin{equation*}
 \frac{n}{s} + \frac{2}{p} =1.
\end{equation*}
Let $(x_{0}, t_{0})$ be arbitrary, where $0<t_{0}<T$.  Let $\sigma \in (0, \sqrt{t_{0}})$. Then, for any $t\in (t_{0}-\sigma^2, t_{0})$ we have 
\begin{equation}
 \int_{t_{0}-\sigma^2}^{t_{0}}\left(\int_{B(x_{0}, \sigma)\cap M_{t}}\abs{A}^{s} \right)^{\frac{2}{s-n}} \leq 
 \int_{t_{0}-\sigma^2}^{t_{0}} (f(t_{0}))^{\frac{2s}{s-n}} =\sigma^2 (f(t_{0}))^{\frac{2s}{s-n}}.
\end{equation}
Fix $\sigma$ so that $ \sigma^2 (f(t_{0}))^{\frac{2s}{s-n}} = \e_{0}(M_{0}, p, q)$ where $\e_{0}(M_{0}, p, q)$ is the small constant in 
Remark \ref{lemmapq}. For this choice of $\alpha$,
the estimate (\ref{eq-pq}), taking $\delta = \frac{\sigma}{2}$ gives
\begin{equation*}
 \abs{A}^2 (x_{0}, t_{0}) \leq \frac{4}{\sigma^2} = \frac{4(f(t_{0}))^{\frac{2s}{s-n}}}{\sigma^2 (f(t_{0}))^{\frac{2s}{s-n}}} =
\frac{4 (f(t_{0}))^{\frac{2s}{s-n}}}{\e_{0}(p, q)}.
\end{equation*}
Taking the supremum of the left hand side with respect to $x_{0}$ and in view of Huisken's estimate \cite{Huisken90} on the lower bound
of $\sup_{M_{t_{0}}}\abs{A}^2$, we get
\begin{equation*}
 \frac{2}{T-t_{0}}\leq \sup_{M_{t_{0}}}\abs{A}^2 \leq \frac{4 (f(t_{0}))^{\frac{2s}{s-n}}}{\e_{0}(p, q)}.
\end{equation*}
This gives the desired inequality. 
\end{proof}
\h The third application is a regularity result without a smallness condition. We will 
use the curvature estimate in Theorem \ref{CS} to obtain other curvature estimates without any smallness condition for mean curvature flow
of surfaces. Our 
result in this direction states
\begin{cor}
Suppose $\mathcal{M} = (M_{t})$ is a smooth, properly embedded
solution of the mean curvature flow in $B(x_{0}, 4\sigma)\times (t_{0}-(4\sigma)^2, t_{0})\subset \RR^{3}\times (t_{0}-(4\sigma)^2, t_{0})$ 
which reaches $x_{0}$ and time $t_{0}$. 
Given a constant $C_{I}>0$, there is a constant  $C_{P}>0$ so that if
\begin{equation}
\label{eq-four}
 \int_{t_{0}-(4\sigma)^2}^{t_{0}} \int_{B(x_{0}, 4\sigma)\cap M_{t}} \abs{A}^4 d\mu dt \leq C_{I}
\end{equation}
then
\begin{equation}
 \sup_{t\in [t_{0}-\sigma^2, t_{0})}\sup_{x\in B(x_{0}, \sigma)\cap M_{t}} \abs{A}^2 (x,t)\leq C_{P}\sigma^{-2}.
\end{equation}
\label{bigness}
\end{cor}
\begin{proof}[Proof of Corollary \ref{bigness}]
We start with the following claim.
\begin{claim}
\label{claim-finite}
Given (\ref{eq-four}) there is a constant $C$ so that
$$\sup_{t\in [t_0-(2\sigma)^2, t_0)}\int_{B(x_0, 2\sigma)\cap M_t} |A| ^2 \, dx \le C.$$
\end{claim}

\begin{proof}
Lets $\eta(x,t)$ be a cut off function compactly supported in $B(x_0,4\sigma)\cap M_t \times [t_0-(4\sigma)^2, t_0)$, identically equal 
to one on $B(x_0,2\sigma)\times [t_0-(3\sigma)^2,t_0)$ (the same one that Ecker used in \cite{EckerCalvar}). Multiply the evolution equation of $|A|^2$
$$\frac{\partial}{\partial t}|A|^2 = \Delta |A|^2 - 2|\nabla A|^2 + 2|A|^4,$$
by $\eta^2$ and integrate it over $M_t$. Using the evolution equation of the volume form $\frac{d}{dt}\mu = -H^2 d\mu$, we see that
\begin{eqnarray}
 \frac{d}{dt}\int_{M_{t}} \abs{A}^2 \eta^2 &\leq& \int_{M_{t}}  \frac{d}{dt} \abs{A}^2 \eta^2 + \abs{A}^2 \frac{d}{dt} \eta^2\nonumber\\ 
& =& \int_{M_{t}}
(\Delta |A|^2 - 2|\nabla A|^2 + 2|A|^4)\eta^2 + \abs{A}^2 \frac{d}{dt} \eta^2\nonumber \\&=&
\int_{M_{t}} \abs{A}^2 2\eta (\frac{d}{dt} - \Delta)\eta + 2\abs{A}^4 \eta^2 + 2\abs{A}^2\eta \Delta \eta + \Delta\abs{A}^2 \eta^2 -2\abs{\nabla A}^2
\eta^2. 
\label{cutoff1}
\end{eqnarray}
Integrating by parts gives
\begin{eqnarray}
\int_{M_{t}} 2\abs{A}^2\eta \Delta \eta + \Delta\abs{A}^2 \eta^2 -2\abs{\nabla A}^2 \eta^2&=& \int_{M_{t}} - 2\nabla (\abs{A}^2\eta)\nabla\eta
-\nabla\abs{A}^2 \nabla \eta^2 -2\abs{\nabla A}^2 \eta^2 \nonumber \\ &=&
\int_{M_{t}} -6 \abs{A}\nabla \abs{A}\eta \nabla\eta - 2\abs{A}^2\abs{\nabla\eta}^2 - 2\abs{\nabla A}^2 \eta^2.
\label{cutoff2}
\end{eqnarray}
Using Kato's inequality $\abs{\nabla \abs{A}}\leq \abs{\nabla A}$ and Cauchy-Schwarz's inequality, one deduces from (\ref{cutoff2}) that
\begin{equation}
 \int_{M_{t}} 2\abs{A}^2\eta \Delta \eta + \Delta\abs{A}^2 \eta^2 -2\abs{\nabla A}^2 \eta^2 \leq \int_{M_{t}} 2\abs{A}^2\abs{\nabla \eta}^2.
\label{cutoff3}
\end{equation}
Combining (\ref{cutoff1}) and (\ref{cutoff3}), we get
\begin{equation*}
 \frac{d}{dt}\int_{M_{t}} \abs{A}^2 \eta^2 \leq \int_{M_{t}}\abs{A}^2 2\eta (\frac{d}{dt} - \Delta)\eta + 2\abs{A}^2\abs{\nabla \eta}^2 + 
2\abs{A}^4 \eta^2.
\end{equation*}
Using that
$$\sup_{M\times [t_0 - 1, t_0]}(\eta^2 + |\nabla\eta|^2 + 2\eta|(\frac{\partial}{\partial t} - \Delta)\eta|) \le \frac{c}{\sigma^2},$$
and that $\vol(B(x_0, 4\sigma)\cap M_t) \le  C\sigma^2$ (this can be proved using Huisken's monotonicity formula \cite{Huisken90}; see for 
example Lemma 1. 4 in \cite{EckerCalvar}) we have
$$\frac{d}{dt}\int_{B(x_0, 4\sigma)}  |A|^2\eta^2 \, d\mu \le \frac{C}{\sigma^2}\int_{B(x_0,4\sigma)\cap M_t} |A|^2\, d\mu + C
\int_{B(x_0,4\sigma)\cap M_t} |A|^4\, d\mu.$$
Choose a cut off function $\psi(t)$ in time so that $\psi = 0$ for $t\in [0, t_0 - (4\sigma)^2]$, $\psi(t) = 1$ for $t\ge t_0 - (2\sigma)^2$ and in between grows linearly. Multiply the previous inequality by $\psi(t)$ and integrate   it over $[t_0 - (4\sigma)^2, t]$, where $t \ge t_0 - (2\sigma)^2$. Then,
$$\int_{B(x_0,2\sigma)\cap M_t} |A|^2\, d\mu \le \frac{C}{\sigma^2}\int_{t_0-(4\sigma)^2}^{t_0}\int_{B(x_0,4\sigma)\cap M_t}|A|^2\, d\mu + \tilde{C}.$$ 
By H\"older inequality and the euclidean volume growth we have
\begin{eqnarray*}
\int_{B(x_0,2\sigma)\cap M_t} |A|^2\, d\mu &\le& \frac{C}{\sigma^2}\left(\int_{t_0-(4\sigma)^2}^{t_0}\int_{B(x_0,4\sigma)\cap M_t} |A|^4\, 
d\mu\right)^{\frac{1}{2}} \cdot\left( \int_{t_0-(4\sigma)^2}^{t_0}\int_{B(x_0,4\sigma)\cap M_t} \, d\mu\right)^{\frac{1}{2}} + \tilde{C} \\
&\le& \frac{C C_I}{\sigma^2}\cdot (16\sigma^2\cdot c\sigma^2)^{\frac{1}{2}} + \tilde{C}\\
&=& \tilde{C},
\end{eqnarray*}
where $\tilde{C}$ is a uniform constant, independent of $\sigma$.
\end{proof}
Having (\ref{eq-four}) and Claim \ref{claim-finite} we can continue as follows.
Let $\e\in (0, \e_{0})$ be a small number to be determined. Here $\e_{0}$ is as in Theorem \ref{CS}. 
Let $N$ be an integer greater than $C_{I}/\e$. Given $x\in B(x_{0}, \sigma)\cap M_{t}$ where 
$t\in [t_{0}-\sigma^2, t_{0})$, there exists $1\leq j\leq N$ with
\begin{equation*}
  \int_{t_{0}-(2\sigma)^2}^{t_{0}} \int_{B(x, 9^{1-j}\sigma)\backslash B(x, 9^{-j}\sigma)\cap M_{t}} \abs{A}^4 d\mu dt \leq C_{I}/N \leq \e \leq \e_{0}.
\end{equation*}
Note that, if $s= 9^{-j}\sigma$ then $B(x, 9s)\subset B(x_{0}, 2\sigma)$. Therefore
\begin{equation}
  \int_{t_{0}-(2\sigma)^2}^{t_{0}} \int_{B(x, 9s)\cap M_{t}} \abs{A}^4 d\mu dt \leq  
\int_{t_{0}-(2\sigma)^2}^{t_{0}} \int_{B(x_{0}, 2\sigma)\cap M_{t}} \abs{A}^4 d\mu dt \leq C_{I}.
\end{equation}
From the estimate
\begin{equation*}
  \int_{t_{0}-(2\sigma)^2}^{t_{0}} \int_{B(x, 9s)\backslash B(x, s)\cap M_{t}} \abs{A}^4 d\mu dt \leq \e\leq \e_{0}
\end{equation*}
we have, by the Choi-Schoen type estimate in Theorem \ref{CS}
\begin{multline}
 \sup_{t\in [t_{0}-(2\sigma-s)^2, t_{0})}\sup_{y\in B(s, 8s)\backslash B(x, 2s)\cap M_{t}} \abs{A}^2 (y,t)\\
\leq \e_{0}^{-1/2} s^{-2}\left( \int_{t_{0}-(2\sigma)^2}^{t_{0}} \int_{B(x, 9s)\backslash B(x, s)\cap M_{t}} \abs{A}^4 d\mu dt\right)^{1/2}
\leq \e_{0}^{-1/2}\e^{1/2} s^{-2}.
\end{multline}
Moreover, inspecting the proof, we can replace extrinsic balls by intrinsic balls $\mathcal{B}(x,s)$. Thus, for each time 
slice $t \in [t_{0}-(2\sigma-s)^2, t_{0}) $, we have the following two estimates
\begin{equation}
 \sup_{y\in \mathcal{B}(s, 8s)\backslash \mathcal{B}(x, 2s)\cap M_{t}} \abs{A}^2 (y,t) \leq \e_{0}^{-1/2}\e^{1/2} s^{-2}
\label{annulus}
\end{equation}
 and
\begin{equation}
 \int_{\mathcal{B}(x, 9s)} \abs{A}^2 (t) d\mu \leq \int_{B(x_{0}, 2\sigma)\cap M_{t}} \abs{A}^2 d\mu\leq C_{I}.
\label{bigball}
\end{equation}
Now, arguing as in the proof of Colding-Minicozzi \cite{CMbig}, Lemma 1.10, one can find a small number $\e$ depending only on $C_{I}$ and $\e_{0}$
such that (\ref{annulus}) and (\ref{bigball}) imply the 
following curvature estimate 
\begin{equation*}
 \sup_{\mathcal{B}(x,s)\subset M_{t}} \abs{A}^2 \leq s^{-2} = (9^{-j}\sigma)^2 \leq 9^{2N}\sigma^{-2}.
\end{equation*}
Hence, for $x\in B(x_{0}, \sigma)\cap M_{t}$ where 
$t\in [t_{0}-\sigma^2, t_{0})$, the following estimate holds
\begin{equation*}
 \abs{A}^{2}(x,t)\leq 9^{2N}\sigma^{-2}.
\end{equation*}
\end{proof}

{} 


\begin{thebibliography}{xx} 
{\small
\bibitem{Brakke} Brakke, K. A. The motion of a surface by its mean curvature. Mathematical Notes, 20. 
Princeton University Press, Princeton, N.J., 1978.
\bibitem{ChenHe} J. Chen; W, He. A note on singular time of mean curvature flow, {\it Math. Z.}
DOI 10.1007/s00209-009-0604-x.
\bibitem{CS}Choi, H. I.; Schoen, R. The space of minimal embeddings of a surface into a three-dimensional manifold of positive Ricci curvature.  
{\it Invent. Math.}  {\bf 81}  (1985),  no. 3, 387--394.
\bibitem{CMMS} Colding, T.H., Minicozzi II, W.P.: Minimal surfaces. Courant Lecture Notes in Math., v. 4, 1999.
\bibitem{CMbig} Colding, T. H.; Minicozzi, W. P., II The space of embedded minimal surfaces of fixed genus in a 3-manifold. II. 
Multi-valued graphs in disks.  {\it Ann. of Math.} (2)  {\bf 160}  (2004),  no. 1, 69--92.
\bibitem{CM}Colding, T. H; Minicozzi II, W. P. Generic mean curvature flow I; generic singularities. http://arxiv.org/abs/0908.3788.
\bibitem{EckerCalvar} Ecker, K. On regularity for mean curvature flow of hypersurfaces.  
{\it Calc. Var. Partial Differential Equations.} {\bf 3}  (1995),  no. 1, 107--126.
\bibitem{Ecker} Ecker, K. Regularity theory for mean curvature flow. Progress in Nonlinear Differential Equations and their 
Applications, 57. Birkhäuser Boston, Inc., Boston, MA, 2004.
\bibitem{EH} Ecker, K.; Huisken, G. Interior estimates for hypersurfaces moving by mean curvature.  Invent. Math.  {\bf 105}  (1991),  no. 3, 547--569.
\bibitem{Huisken84} Huisken, G., Flow by mean curvature of convex surfaces into spheres.
{\it J. Differential Geom.}  {\bf 20}  (1984),  no. 1, 237--266.
\bibitem{Huisken90} Huisken, G. Asymptotic behavior for singularities of the mean curvature 
flow.  J. Differential Geom.  {\bf 31}  (1990),  no. 1, 285--299.
\bibitem{Hu85} Huisken, G. Contracting convex hypersurfaces in Riemannian manifolds by their mean curvature. {\it Invent. Math.} {\bf 84} (1986), 463--480.
\bibitem{HS} Huisken, G.; Sinestrari, C. Mean curvature flow singularities for mean convex surfaces.
{\it Calc. Var. Partial Differential Equations.} {\bf 8}  (1999),  no. 1, 1--14.
\bibitem{I1}Ilmanen, T. Singularities of Mean Curvature Flow of Surfaces, preprint, 1995,
http://www.math.ethz.ch/\~/papers/pub.html.
\bibitem{I2} Ilmanen, T. Lectures on Mean Curvature Flow and Related Equations, preprint, 1998,
http://www.math.ethz.ch/\~/papers/pub.html.
\bibitem{Le} Le, N. Q. On the convergence of the Ohta-Kawasaki Equation to motion by nonlocal Mullins-Sekerka Law, preprint.
\bibitem{LS} Le, N. Q., Sesum, N. On the extension of the mean curvature flow, to appear in {\it Math. Z. }.
\bibitem{Schatzle1} Sch\"{a}tzle, R. Lower semicontinuity of the Willmore functional for currents.  
{\it J. Differential Geom.}  {\bf 81}  (2009),  no. 2, 437--456.
\bibitem{SZ} Shen, Y.-B.; Zhu, X.-H. On stable complete minimal hypersurfaces in $R\sp {n+1}$.  {\it Amer. J. Math.}  {\bf 120}  (1998), 
 no. 1, 103--116. 
\bibitem{S} Simon, L. {\it Lectures on geometric measure theory}; Proc. of the Centre for Math. Analysis, Austr.Nat.Univ., Vol. 3, (1983).
\bibitem{Sm} Smoczyk, K.,  Starshaped hypersurfaces and the mean curvature flow.  {\it Manuscripta Math.}  {\bf 95}  (1998),  no. 2, 225--236.
\bibitem{St} Stone, A. A density function and the 
structure of singularities of the mean curvature flow.  {\it Calc. Var. Partial Differential Equations}  {\bf 2}  (1994),  no. 4, 443--480.
\bibitem{WhiteCrelle}White, B. Stratification of minimal surfaces, mean curvature flows, and harmonic maps. {\it J. Reine
Angew. Math.} {\bf 488} (1997), 1--35.
\bibitem{White} White, B. A local regularity theorem for mean curvature flow.  {\it Ann. of Math.} (2)  {\bf 161}  (2005),  no. 3, 1487--1519.
\bibitem{XYZ1}Xu H. W., Ye, F., Zhao, E. T. Extend Mean Curvature Flow with Finite Integral Curvature, 	arXiv:0905.1167v1.
\bibitem{XYZ2}Xu H. W., Ye, F., Zhao, E. T. The Extension for Mean Curvature Flow with Finite Integral Curvature in 
Riemannian Manifolds, arXiv:0910.2015v1.
}
\end{thebibliography}
\end{document}